\documentclass[11pt]{article}
\usepackage{amsmath}
\usepackage{amssymb}
\usepackage{amsthm}
\usepackage{lineno}
\usepackage{graphicx}
\usepackage{caption}
\usepackage{subcaption}
\usepackage{blkarray}
\usepackage{enumerate}
\usepackage{enumitem}
\usepackage[font=small,skip=0pt]{caption}
\usepackage{xcolor}
\usepackage{verbatim}
\usepackage{geometry}
\usepackage{enumerate}
\usepackage{wrapfig}
\usepackage{color,soul}
\usepackage{float}
\usepackage{lineno}
\usepackage[stable]{footmisc}
\usepackage{lineno}
\usepackage{blindtext}
\usepackage{calc}
\usepackage{upgreek}
\usepackage{amsmath}
\usepackage{amssymb}
\usepackage{amsthm}
\usepackage{lineno}
\usepackage{graphicx}
\usepackage{caption}
\usepackage{subcaption}
\usepackage[font=small,skip=0pt]{caption}
\usepackage{xcolor}
\usepackage{verbatim}
\usepackage{geometry}
\usepackage[colorinlistoftodos]{todonotes}
\usepackage{wrapfig}
\usepackage{color,soul}
\usepackage{float}
\usepackage{url}
\usepackage[colorinlistoftodos]{todonotes}
\usepackage{wrapfig}
\usepackage{color,soul}
\usepackage{float}
\usepackage{multirow}

\newtheorem{Theorem}{Theorem}
\newtheorem{Corollary}{Corollary}

\newtheorem{Lemma}{Lemma}

\title{Distributional properties of fluid queues busy period and first passage times}
\author{Zbigniew Palmowski \footnote{Corresponding author e-mail address: zbigniew.palmowski@gmail.com. This work is partially supported by Polish National Science Centre Grant
No. 2016/23/B/HS4/00566 (2017-2020).}
       }

\begin{document}
\maketitle

\begin{abstract}
In this paper we analyze the distributional properties of a busy period in an on-off
fluid queue and the a first passage time
in a fluid queue driven by a finite state Markov process.
In particular, we show that in Anick-Mitra-Sondhi model the first passage time has a IFR distribution and the busy period has a DFR distribution.
\end{abstract}


\newcommand{\halmos}{{\mbox{\, \vspace{3mm}}} \hfill
\mbox{$\Box$}}

\section{Introduction}

Fluid queues are used to represent systems where some quantity accumulates or is depleted; gradually over time, subject to some random (usually Markov) environment.
The stochastic fluid queues offer powerful modeling ability for a wide range of real-life systems of significance.
They are highly used in performance evaluation of telecommunication systems, modelling dams and health systems, in
environmental sciences, biology, physics, etc; see \cite{Kulkarni} for overview and references therein.
Existing theoretical frameworks of analysing this type of queues come
mainly from stochastic processes theory and
so-called analytic methods.
In this paper we consider two kinds of fluid queues.

The first fluid queue analyzed in this paper is so-called {\it on-off} fluid queue.
In this queue we have $N$ sources.
By an {\it on-off} $i$th flow ($i=1,2,\ldots,N$)
we  mean a $0-1$ process $\xi_i(t)$, in which
consecutive off-periods alternate with on-periods.
Random on and off times are independent of each other.
In this paper we assume that the silence periods of $i$th source ($i=1,2,\ldots,N$)
are exponentially distributed with parameter $\lambda_i$, whereas the activity periods are generally
distributed with a distribution function $A_i$ which is absolutely continuous with the density function $a_i$
having finite mean $\int_0^\infty xa_i(x)\; dx<+\infty $.
The input rate $r_i$ of the $i$th source
is larger than the output rate of
the infinite buffer which is assumed to be $1$.
Then the input rate to the buffer at time $t$ equals
$Z(t)=\sum_{\ell=1}^N r_\ell \xi_\ell(t)$
and the constant output rate is $1$,
wherein
the {\em buffer content process } $Q(t)$ is governed by the equation
\begin{eqnarray*}
\frac{d Q(t)}{dt}=\left\{
\begin{array}{ll}
Z(t)-1, & \mbox{for $Q(t)>0$}\\
(Z(t)-1)_+, & \mbox{for $Q(t)=0$.}
\end{array}
\right.
\end{eqnarray*}
We assume the stability condition
$$\sum_{i=1}^Nr_i\frac{\int_0^\infty xa_i(x)\; dx}{\int_0^\infty xa_i(x)\; dx+1/\lambda_i}<1$$ under which
the steady-state buffer content is well-defined.
The seminal model is so-called of Anick-Mitra-Sondhi model, where
$r_i=r$ and $A_i$ has exponential distribution; see \cite{AMS} for details.

In the second fluid model, called {\it Markov fluid queue}, the process
$\{(Q(t), J(t)),t\geq 0\}$ consists of continuous level process $Q(t)\geq 0$, a phase variable $J(t)\in\mathcal{I}=\{1,2,\ldots,N\}$ for some fixed $N\in \mathbb{N}$ and real-valued rates $r_i$, $i\in\mathcal{I}$, such that
\begin{itemize}
\item the phase process $\{J(t),t\geq 0\}$ is an irreducible, continuous-time Markov chain with finite state space $\mathcal{I}$;
% and generator ${\bf T}$.
\item when $J(t)=i$ and $Q(t)>0$, the rate of change of $Q(t)$ at time $t$ is given by $r_i$, that is, $\frac{d\;Q(t)}{d\;t} =r_i$;
\item when $J(t)=i$ and $Q(t)=0$, the rate of change of $Q(t)$ at time $t$ is given by $\max\{0,r_i\}$.
\end{itemize}
In this case we assume the following stability condition as well to define the steady-state
buffer content:
$$\sum_{i=1}^N p_i r_i<0,$$
where $p_i$ $(i=1,2,\ldots,N$) denotes the stationary distribution of $J(t)$.

The main goal of this article is to analyze the distributional
properties of the busy periods and the first passage times in the fluid queue theory.
We refer to \cite{Barlow} and \cite[Sec. 1.4 and 1.7]{Rysiek} to the properties of DFR and IFR distributions
and to the distributions with completely monotone density functions.

Firstly, in next section, we prove that in the first fluid model, under assumption that
activity density function $a_i$ is completely monotone for fixed $i\in\{1,2,\ldots, N\}$,
the busy periods starting with an activity period of $i$th fixed source
has a completely monotone density function. Hence it has Decreasing Failure Rate
(it is so-called a DFR distribution).
We recall that function $f$ satisfying $(-1)^kf^{(k)}(x)\geq 0$, $k=0,1,2,\ldots$, is called completely
monotonic and for a distribution function $F$ its failure rate is defined via
$$r_F(x)=\frac{F^\prime(x)}{1-F(x)}.$$
Hence, in particular, the busy period in the Anick-Mitra-Sondhi model has
a DFR distribution.
It is worth to add that the fact that a density is completely monotone means
that it is the limit (in a weak sense) of some sequence of finite mixtures of exponentials and
hence it is infinitely divisible;
see \cite{Steutel} for details.

Later we prove that for the second fluid queue, the first passage time
over fixed level of the buffer content process $Q(t)$
has a Increasing Failure Rate (it is so-called a IFR distribution).

Some of the techniques used in this paper are strongly related
with the tools applied in classical queueing theory concerning single server queues, like for example for $GI|GI|1$ queue.
In particular, the key equation \eqref{LTbusy} for the proof of main Theorem \ref{mainthm1}
is a consequence of the work-conservation principle.
For the classical theory of queues, in particular for the analysis of the busy period of single server queue,
we refer to \cite{Cohen}.

Distributional properties of the busy period and first passage times have been discussed before by a number of authors
in queueing theory.
Notable amongst these are \cite{1,2,Keilson,4,5,6,15}
in the context of Markov chains,
\cite{7,8,9,10,11,12,13,14, 16,18} related strongly
to fluid queues and \cite{17} that put fluid queues
into the framework of Markov additive processes.

\section{Main results}\label{sec:main}
\subsection{Busy period}\label{sec:busy}

Let consider classical on-off model, where we have $N$ on-off sources with silence times of $i$ th source which are exponential
distributed with intensity $\lambda_i$ and activity periods of $i$th source have completely monotone generic distribution $A_i$
which is absolutely continuous with the density function $a_i$ and
with the finite mean and with the Laplace transform $$\alpha_i(\theta)=\int_0^\infty e^{-\theta x}a_i(x)\;dx.$$
Source $i$ constantly transmits at rate $r_i\geq 1$ when active and the total output rate from the buffer equals $1$.
Note that silence periods are exponentially distributed with the intensity $\lambda=\sum_{i=1}^N\lambda_i$.
Denote by $P_i$ the busy period that starts from activity of $i$th source and let $b_i$ will be its density function
with Laplace transform $\pi_i(\cdot)$.

Our first main result is the following theorem.
\begin{Theorem}\label{mainthm1}
Assume that for fixed $i\in\{1,\ldots,N\}$ the density $a_i$ of activity of $i$th source
is a completely monotone function. Then the busy period $P_i$
starting from activity of $i$th source has a completely monotone density function $b_i$.
\end{Theorem}
By \cite[Prop. 5.9A, p. 75]{Keilson} we can conclude the following crucial fact.
\begin{Corollary}
$P_i$ has a DFR distribution.
\end{Corollary}
Before we give the proof of the main results we start from proving the following fact.
\begin{Lemma}\label{lemmathm1}
Assume that some probability density functions $K$ and $L_i$ ($i=1,\ldots, N$) are completely monotone with Laplace transforms $k(\cdot)$ and $l_i(\cdot)$, respectively.
Then the random variable with the Laplace transform:
\[\beta(\theta)=k\left[\theta + \sum_{i=1}^Nc_i(1-l_i(\theta))\right],\]
where $c_i\geq 0$, has a completely monotone density function.
\end{Lemma}
\begin{proof}
We will now mimic the idea presented in the proof of \cite[Thm. 2.1]{Keilson2}.
Denoting $c=\sum_{j=1}^Nc_j$,
\[ \beta(\theta)=\sum_{m=0}^\infty\int_0^\infty e^{-\theta x}e^{-cx} \frac{(x\sum_{j=1}^N c_jl_j(\theta))^m}{m!}K(x)\;dx
\]
is the transform of sums of convolutions of some densities with support on positive half-line.
Moreover, $\beta(0)=k(0)=1$. Hence $\beta(\cdot)$ is the Laplace transform of some random variable having absolutely continuous distribution.
Further, since $K$ is completely monotone, it is the limit (in a weak sense) of some sequence of finite mixtures of exponentials;
see \cite{Steutel}. Thus it suffices to prove
our theorem only when $K$ is an exponential density function with a parameter, say, $s$. In this case
\[\beta(\theta)=s/\left\{\theta +s+c-\sum_{j=1}^N\sum_{i=1}^{K_j}p_{ji}\gamma_{ji}/(\theta+\gamma_{ji})\right\},
\]
where $K_{j}$ are finite and $p_{ji}, \gamma_{ji}>0$ with $\sum_{i=1}^{K_j}p_{ji}=1$. The denominator
$$\theta +s+c-\sum_{j=1}^N\sum_{i=1}^{K_j}p_{ji}\gamma_{ji}/(\theta+\gamma_{ji})$$
is an increasing function of $\theta$
(except simple poles $-\gamma_{ji}$ where its is not well-defined). It is than clear that this denominator has a simple zero between each pair of adjacent poles and a zero in the intervals
$(-\infty, -\max_{j,i}\gamma_{ji})$ and $(-\min_{j,i}\gamma_{ji}, 0)$.   Hence there are $\sum_{j=1}^NK_j+1$ zeros on the negative real $\theta$ axis.
Since $\beta(\theta)$ is rational with the denominator of order $\sum_{j=1}^NK_j+1$
and this denominator is monotonic,
all residues are positive.
Thus from Cauchy's residue theorem it follows that
the density with the Laplace transform $\beta(\theta)$ is completely monotone
(it is a finite mixture of exponentials).
This observation completes the proof of above statement.
\end{proof}
We are now ready to prove Theorem \ref{mainthm1}.
\begin{proof}
Note that the speed at which the buffer size increases at some time $t$ is completely determined
by which sources are active at $t$.
Thus transmitting fluid from some active sources before any other fixed active source
does not affect the busy period.
In other words, we can give preemptive priority
to all sources $j$ for $j \neq i$ over fixed source
$i$.
Hence using the work-conservation arguments we can conclude that,
as far as busy periods are concerned, our initial fluid model is equivalent to
the fluid model with $i$th source and
one source replacing all sources $j\neq i$
with preemptive priority over this $i$th source.
In this way, the sources $j \neq i$ do not see the $i$th source
and therefore they generate their own busy periods.
Moreover, the source $i$ is served only during the corresponding
idle periods. This argument produces the following key system of equations for the Laplace transforms of the busy periods
\begin{equation}\label{LTbusy}
\pi_i(\theta)=\alpha_i\left[r_i\theta +\lambda_i(r_i-1)(1-\pi_i(\theta))+\sum_{j\neq i} \lambda_jr_i(1-\pi_j(\theta))\right].
\end{equation}
We refer to \cite[Thm. 5.2]{OnnoDumas} for detailed proof of this fact.

Consider now the sequence of distribution functions with positive support whose Laplace transforms are generated recursively by:
\[\pi_{i,0}(\theta)=1,\]
\begin{equation}\label{odnosnik}\pi_{i,m+1}(\theta)= \alpha_i\left[r_i\theta +\lambda_i(r_i-1)(1-\pi_{i,m}(\theta))+\sum_{j\neq i} \lambda_jr_i(1-\pi_{j,m}(\theta))\right].\end{equation}
Because of the positive support and finite first moment of $A_i$, Laplace transform $\alpha_i$ decreases with $\theta$ and $\alpha_i^\prime(\theta)<0$ on $(0,\infty)$.
Hence \[\pi_{i,m+1}(\theta)-\pi_{i,m}(\theta)=\int_{r_i\theta +\lambda_i(r_i-1)(1-\pi_{i,m-1}(\theta))+\sum_{j\neq i} \lambda_jr_i(1-\pi_{j,m-1}(\theta))}^{r_i\theta +\lambda_i(r_i-1)(1-\pi_{i,m}(\theta))+\sum_{j\neq i} \lambda_jr_i(1-\pi_{j,m}(\theta))}\alpha_i^\prime(u)\;du<0.\]
Thus $\pi_{i,m}(\theta) < 1$ because $\pi_{i,0}(\theta)=1$. Moreover, we have that $r_i<1$. Hence
the argument of $\alpha_i$ in \eqref{odnosnik} can be estimated from above by
$$r_i\theta +\lambda_i(r_i-1)(1-\pi_{i,m}(\theta))+\sum_{j\neq i} \lambda_jr_i(1-\pi_{j,m}(\theta))
\leq r_i\theta+\sum_{j\neq i} \lambda_j\leq r_i\theta+\lambda$$
for $\lambda =\sum_{j=1}^N \lambda_j$.
This gives that $\pi_{i,m+1}(\theta)\geq \alpha_i(r_i\theta+\lambda)$ and therefore the sequence $\pi_{i,m}$
is uniformly bounded and so decreases to $\pi_i$ solving (\ref{LTbusy}).
Now Lemma \ref{lemmathm1} completes the proof.
\end{proof}
%\halmos

\subsection{First passage times}\label{sec:fp}
We will consider now the second fluid queue model with the random environment process $J(t)$ on the finite state space $\mathcal{I}=\{1,\ldots, N\}$.
Recall that $r_i$ is a netput intensity when $J(t)=i$. We order set $\mathcal{I}$ in such way that $r_i\leq r_j$ as $i< j$.
We will discretize the time and consider Markov chain $X_n =(Q(\frac{n}{m}), J(\frac{n}{m}))$ where $Q(t)$ is a buffer
content process at time $t$. We will consider partial ordering on $R\times \mathcal{I}$ with
$(x,i)\leq (y,j)$ iff $x\leq y$ or $x=y$ and $i\leq j$.

\begin{Theorem}
The first passage time $\tau(x)=\inf\{t\geq 0: Q(t)>x\}$ has increasing failure rate (is IFR).
\end{Theorem}
\begin{proof}
In the first step we prove that
$\tau^m(x)=\min\{n\in N: X_n\geq (x,1)\}$
is IFR and then we approximate $\tau(x)$ by $\tau^m(x)$ as $m\to\infty$.
To prove the first statement by \cite[Th. 3.1]{DLP} it suffices to prove that
$P(X_1\leq z|X_0=w)$
is $TP_2$ in $z$ and $w$ belonging to $R\times \mathcal{I}$ which means that for $x<y$, $a<b$ and $i_3< i_4$ if $a=b$ or $i_1< i_2$ if $x=y$ we have
\begin{equation*}\frac{P(X_1\leq (y, i_2)|X_0=(a,i_3))}{P(X_1\leq (x, i_1)|X_0=(a, i_3))}\leq \frac{P(X_1\leq (y, i_2)|X_0=(b, i_4))}{P(X_1\leq (x, i_1)|X_0=(b, i_4))}\end{equation*}
or that
\begin{equation}\label{TP2}
P(X_1\leq (x, i_1)|X_1\leq (y, i_2), X_0=(a, i_3))\geq P(X_1\leq (x, i_1)|X_1\leq (y, i_2), X_0=(b, i_4)).\end{equation}
Heuristically, the total positive of order two $TP_2$ ordering between distributions
corresponds to ordering of positive dependence; see \cite{DLP}
and \cite{Rysiek} for further details.

Now, observe that the LHS of (\ref{TP2}) equals $1$ if $x>a+r_{i_3}$. Similarly, RHS equals $0$ if $x< b+r_{i_4}$.
For the same reasons $y\geq b+r_{i_4}> a+r_{i_3}$.  Hence the only possible left case is when $a=b$ and
$a+r_{i_4}\leq x\leq a+r_{i_3}$ which is impossible by our assumptions.
To conclude $\tau^m(x)$ is IFR which is equivalent to the statement that $\log P(\tau^m(x)>t)$ is concave.

It suffices thus to prove now that $P(\tau(x)>t)=\lim_{m\to\infty}P(\tau^m(x)>t)$ or that $P(\sup_{s\leq t}Q(s)\leq x)=\lim_{m\to\infty}
P(\sup_{k/m\leq t}X_k\leq x)$. Note that $\sup_{s\leq t}Q(s)-\sup_{k/m\leq t}X_k\leq r_N\frac{1}{m}\rightarrow 0$ as $m\to\infty$.
Moreover, clearly $\sup_{s\leq t}Q(s)$ has density function well-defined. This completes the proof.
\end{proof}
%\halmos

\subsubsection*{Acknowledgements}
Zbigniew Palmowski thanks G.Latouche and G. Nguyen for inspiring discussions.

\end{document}